\documentclass[a4paper,12pt]{article}
\usepackage{amssymb}
\usepackage{amsthm}
\usepackage{tikz}
\usepackage{tikz,pgfplots}
\usepackage{amsmath,amssymb}
\usepackage{todonotes}
\usepackage{url}
\usepackage{diagbox}

\usepackage{color}
\usepackage{graphicx}
\usepackage{placeins,caption}

\DeclareMathOperator{\gp}{gp}

\DeclareMathOperator{\cp}{\,\square\,}
\usepackage[top=3cm,bottom=3cm,left=2.8cm,right=2.8cm]{geometry}
\newcommand{\cupdot}{\mathbin{\mathaccent\cdot\cup}}

	\newtheorem{theorem}{Theorem}[section]
	
	\newtheorem{corollary}[theorem]{Corollary}
	\newtheorem{proposition}[theorem]{Proposition}

	\newtheorem{problem}[theorem]{Problem}
	
	\theoremstyle{definition}

	\newtheorem{definition}[theorem]{Definition}

\newcommand{\address}[1]{#1}
	
\begin{document}
	
		\title{General Position Polynomials}
		\author{Vesna Ir\v{s}i\v{c} $^{a,b}$ \\ \texttt{\footnotesize vesna.irsic@fmf.uni-lj.si}  \\
		{\footnotesize ORCID: 0000-0001-5302-5250}
            \and
            Sandi Klav\v{z}ar $^{a,b,c,}$\footnote{corresponding author, available at sandi.klavzar@fmf.uni-lj.si} \\ \texttt{\footnotesize sandi.klavzar@fmf.uni-lj.si}  \\
		{\footnotesize ORCID: 0000-0002-1556-4744}
            \and
		Gregor Rus $^{a,b}$ \\ \texttt{\footnotesize gregor.rus@fmf.uni-lj.si}  \\
		{\footnotesize ORCID:  0000-0001-9739-0575}
		\and
		James Tuite $^{d}$ \\ \texttt{\footnotesize james.t.tuite@open.ac.uk} \\ 
		{\footnotesize ORCID: 0000-0003-2604-7491}
  }

\maketitle

\address{
	$^a$ Faculty of Mathematics and Physics, University of Ljubljana, Slovenia
	
	$^b$ Institute of Mathematics, Physics and Mechanics, Ljubljana, Slovenia
	
	$^c$ Faculty of Natural Sciences and Mathematics, University of Maribor, Slovenia
	
	$^d$ Department of Mathematics and Statistics, Open University, Milton Keynes, UK

}

\begin{abstract}
A subset of vertices of a graph $G$ is a general position set if no triple of vertices from the set lie on a common shortest path in $G$. In this paper we introduce the general position polynomial as $\sum_{i \geq 0} a_i x^i$, where $a_i$ is the number of distinct general position sets of $G$ with cardinality $i$. The polynomial is considered for several well-known classes of graphs and graph operations. It is shown that the polynomial is not unimodal in general, not even on trees. On the other hand, several classes of graphs, including Kneser graphs $K(n,2)$, with unimodal general position polynomials are presented.
\end{abstract}


\noindent
AMS Subj.\ Class.\ (2020): 05C31, 05C69, 05C76

\section{Introduction}

Given a graph $G = (V(G), E(G))$, a set $S\subseteq V(G)$ of vertices of $G$ is a {\em general position set} if no triple of vertices from $S$ lie on a common shortest path in $G$. The cardinality of a largest general position set of $G$ is called the \emph{general position number} of $G$ and is denoted by $\gp(G)$. These sets were independently introduced in~\cite{Chandran-2016, manuel-2018a} and have already been studied from many perspectives, cf.~\cite{AnaChaChaKlaTho, KlaKriTuiYer, klavzar-2022, KorzeVesel, Patkos-2019, tian-2021, yao-2022}. In this paper, we explore general position sets from the point of view of the counting polynomial defined in the following standard way. 
 
\begin{definition}
The \emph{general position polynomial} of a graph $G$ is the polynomial 
$$ \psi (G) = \sum_{i \ge 0}a_ix^i\,,$$ 
where $a_i$ is the number of distinct general position sets of $G$ with cardinality $i$.
\end{definition}

A polynomial is said to be {\em unimodal} if its coefficients are non-decreasing and then non-increasing. Unimodality is one of the most important and most studied properties of counting polynomials in graph theory. In the next paragraph, we give a very brief justification for our claim.  

Since the {\em matching polynomial} of a graph has only real zeros~\cite{Heilmann-1972}, it is unimodal. The unimodality of the {\em chromatic polynomial} has been established in~\cite{Huh-2012}. In~\cite{AlikhaniPeng-2014} it has been conjectured that the {\em domination polynomial} of an arbitrary graph $G$ is also unimodal. The conjecture has been approached from different perspectives, see~\cite{AlikhaniJahari-2017, Beaton-2022, Burcroff-2023, LauAlikhani-2022}, but it remains open. On the other hand, it is known that the {\em independence polynomial} is not unimodal in general, but it has been conjectured by Alavi, Malde, Schwenk, and Erd\H{o}s that the independence polynomial of a tree is unimodal~\cite{Alavi-1987}, a conjecture which is also still open. It was very recently demonstrated that the conjecture cannot be strengthened up to its log-concave version~\cite{kadrawi-2023+}. On the other hand, the independence polynomial of a claw-free graph is unimodal~\cite{Chudnovsky-2007}.  

The rest of the paper is organised as follows. In the next section we determine the general position polynomial of several families of graphs and give an inclusion-exclusion-like formula for the polynomial. We also construct an infinite number of pairs of non-isomorphic trees with the same general position polynomial. In Section~\ref{sec:operations} we consider the general position polynomials of disjoint unions of graphs, joins of graphs, and Cartesian products of graphs. In particular, we express the general position polynomial of the join of two graphs with the clique polynomial and the independent union of cliques polynomial (to be defined in Section~\ref{sec:operations}) of the factors, and determine $\psi(P_r\cp P_s)$. Then, in Section~\ref{sec:unimodality}, we consider unimodality of the polynomial. We first demonstrate that it is not unimodal in general and not even unimodal on the class of trees. On the other hand, we prove that it is unimodal on combs, Kneser graphs $K(n,2)$, and a family of graphs containing complete bipartite graphs minus a matching. The paper is concluded with some open problems and suggestions for future research. 

\section{Basic results and examples}
\label{sec:basics}

Let $G$ be a graph. Then, clearly, the degree of $\psi (G)$ is $\gp (G)$. It was shown in~\cite[Theorem 2.10]{Chandran-2016} that $C_4$ and $P_n$, $n\ge 2$, are the only connected graphs $G$ with $\gp(G) = 2$, hence among connected graphs the degree of a general position polynomial is equal to $2$ precisely for $C_4$ and for $P_n$, $n\ge 2$. In addition, if $G$ is of order $n$, then since every set of at most two vertices is a general position set, its general position polynomial starts as 
\begin{equation}
\label{eq:first-3-terms}
\psi (G) = 1 + nx + {n \choose 2}x^2 + \cdots 
\end{equation}

We now derive general position polynomials for some standard families of graphs. 

\begin{proposition}
\label{prop:examples}
\begin{enumerate}
\item[(i)] If $n\ge 1$, then $\psi(K_n) = (1+x)^n$.
\item[(ii)] If $n\ge 1$, then $\psi (P_n) = 1+nx+{n \choose 2}x^2$.
\item[(iii)] If $n\ge 3$ is odd, then $\psi(C_{n}) = 1+nx+{n \choose 2}x^2+ \left ( {n \choose 3}-n{\lfloor \frac{n}{2} \rfloor \choose 2}\right ) x^3$. 
\item[(iv)] If $n\ge 4$ is even, then $\psi (C_{n}) = 1+nx+{n \choose 2}x^2+\left ( {n \choose 3}-n{{\frac{n}{2} -1 }\choose 2}-\frac{n(n-2)}{2}\right ) x^3$. 
\item[(v)] If $m\geq n\ge 1$, then $\psi (K_{m,n}) = 1+ (m+n)x+{m+n \choose 2}x^2 + \sum_{i=3}^m \left({m \choose i} + {n \choose i}\right)x^i$.
\end{enumerate}
\end{proposition}

\begin{proof}
(i) Any subset of $i$ vertices in $K_n$ for $0 \leq i \leq n$ is in general position, so the coefficient at $x^i$ in $\psi (G)$ is ${n \choose i}$. Thus $\psi (G) = \sum _{i=0}^{n}{n \choose i}x^i = (1+x)^n$.

(ii) Follows from~\eqref{eq:first-3-terms} and the previously mentioned fact that $\gp(P_n) = 2$ for $n\ge 2$. 

(iii) Consider $C_{2d+1}$, $d\ge 1$. We count the number of triples of vertices that are on a common geodesic. For $2 \leq r\leq d$ there are exactly $n$ pairs of vertices at distance $r$ on the cycle and each such pair corresponds to exactly $r-1$ sets from $\binom{V(C_n)}{3}$ that are on a common geodesic. Thus there are $n\sum_{r=1}^{d} (r-1) = n{d \choose 2}$ triples that are on a common geodesic and hence there are exactly ${n \choose 3}-n{d \choose 2}$ triples that are in general position.   

(iv) Consider $C_{2d}$, $d\ge 2$. The reasoning for odd cycles applies to pairs of vertices at distance at most $d-1$ from each other; however, each pair of vertices at distance $d$ now corresponds to $n-2$ sets from $\binom{V(C_n)}{3}$ on a common geodesic. Therefore there are 
\[ \frac{n}{2}(n-2)+n \sum_{r=1}^{d-1} (r-1) = n{d-1 \choose 2}+\frac{n}{2}(n-2)\] 
triples of vertices that are on a common geodesic.

(v) The formula follows since $\gp(K_{m,n}) = \max\{m, n\} = m$ and since a general position set $S$ of $K_{m,n}$ of cardinality at least $3$ is an independent set, so that $S$ is a subset of one of the bipartition sets of $K_{m,n}$. 
\end{proof}

The general position polynomial can also be expressed via the inclusion-exclusion principle as follows. For a positive integer $n$, let $[n] = \{1,\ldots,n\}$.

\begin{proposition}
\label{prop:inclusion-exclusion}
Let $G$ be a graph and let $X_1,\ldots,X_n$ be the maximal general position sets of $G$. Then 
$$\psi(G) = \sum_{k = 1}^n  (-1)^{k-1} \sum_{\{i_1,\ldots,i_k\}\subseteq [n]} \psi(X_{i_1} \cap \cdots \cap X_{i_k})\,.$$
\end{proposition}

\begin{proof}
Any subset of a general position set $X$ is also a general position set and the number of subsets of size $i$ is ${|X| \choose i}$. It follows that for every general position set $X$ we have $\psi(X) = (1+x)^{|X|}$. The formula then follows by the inclusion-exclusion principle. 
\end{proof}

As an example, consider the Petersen graph $P = K(5,2)$. In the standard drawing of it denote the consecutive vertices of the outer $5$-cycle by $u_0, u_1, u_2, u_3, u_4$, and their respective neighbors on the inner $5$-cycle by $u_0, u_1, u_2, u_3, u_4$. It is known from the seminal paper~\cite{manuel-2018a} that $\gp(P) = 6$. By inspection it can be checked that there are precisely five general position sets of cardinality $6$:  
\begin{align*}
& \{u_0,u_1,u_3,v_2,v_3,v_4\}, 
\{u_0,u_2,u_3,v_0,v_1,v_4\}, 
\{u_0,u_2,u_4,v_1,v_2,v_3\}, \\
& \{u_1,u_2,u_4,v_0,v_3,v_4\}, 
\{u_1,u_3,u_4,v_0,v_1,v_2\}\,.
\end{align*}
Moreover, the remaining maximal general position sets are the five independent sets of cardinality 4: 
\begin{align*}
& 
\{u_0,u_2,v_3,v_4\}, 
\{u_0,u_3,v_1,v_2\}, 
\{u_1,u_3,v_0,v_4\}, 
\{u_1,u_4,v_2,v_3\}, 
\{u_2,u_4,v_0,v_1\}\,.
\end{align*}
Since every vertex of $P$ lies in five different maximal general position sets, the intersection of six or more such sets is empty. In Table~\ref{tab:intersections} it is shown how many different occurrences of the same number of sets in an intersection have the same size of intersection.

\begin{table}[ht!]
\begin{center}
\begin{tabular}{|c||c|c|c|c|c|c|}\hline
\diagbox[width=10em]{no.\ of\\ maximal sets}{intersection size}&
0 & 1 & 2 & 3 & 4 & 5 \\ \hline\hline
2 & 5 & 10 & 0 & 30 & 0 & 0 \\ \hline
3&50&40&30&0&0&0\\\hline
4&160&50&0&0&0&0\\\hline
5&242&10&0&0&0&0\\\hline
6&210&0&0&0&0&0\\\hline
7&120&0&0&0&0&0\\\hline
8&45&0&0&0&0&0\\\hline
9&10&0&0&0&0&0\\\hline
10&1&0&0&0&0&0\\\hline
\end{tabular}
\caption{Number of different occurrences of the same number of maximal general position sets in an intersection having the same size of intersection.}
\label{tab:intersections}
\end{center}
\end{table}

Combining Proposition~\ref{prop:inclusion-exclusion} with Table~\ref{tab:intersections} yields: 
\begin{align*}
\psi(P) =\ & \left[5(x+1)^6+5(x+1)^4)\right] - \left[5(x+1)^0 + 10(x+1)^1 + 30(x+1)^3)\right] \\
& +\left[50(x+1)^0+40(x+1)^1+30(x+1)^2\right] - \left[160(x+1)^0+50(x+1)^1\right] \\ 
& + \left[242(x+1)^0+10(x+1)^1\right] -210+120-45+10-1 \\ 
=\ & 1 + 10x + 45x^2 + 90x^3 + 80x^4 + 30x^5+5x^6.
\end{align*}

To conclude the section we point out that the general position polynomial does not determine a graph uniquely. For example, $1 + 4 x + 6 x^2$ is a general position polynomial of both $P_4$ and $C_4$. Furthermore, the general position polynomial does not even determine a tree uniquely. For example, let $k \in \mathbb{N}$ and take $T_1^{(k)}$ to be the tree obtained from identifying one leaf of $P_{13k}$, $P_{5k}$ and $P_{4k}$, and $T_2^{(k)}$ to be the tree obtained from identifying one leaf of $P_{10k}$, $P_{9k}$ and $P_{3k}$. The case $k=1$ is shown in Fig.~\ref{fig:two-trees}.

\begin{figure}[ht!]
    \centering
    \begin{tikzpicture}
    [scale=0.5,
    vert/.style={circle,fill=black,draw=black, inner sep=0.05cm},
    s/.style={fill=black!15!white, draw=black!15!white, rounded corners},
    outvert/.style={rectangle,draw=black},
    outedge/.style={line width=1.5pt},
    dom_vert/.style={circle,draw=black,fill=white, inner sep=0.05cm}
    ] 
        \begin{scope}[xshift=0cm]
        \foreach \x in {1,...,12}
        \node[vert] (\x) at (\x, 0) {};
        \foreach \x [remember=\x as \lastx (initially 1)] in {2,...,12}
        \path (\x) edge (\lastx);

        \foreach \x in {14,...,17}
        \node[vert] (\x) at (\x, 0) {};
        \foreach \x [remember=\x as \lastx (initially 14)] in {15,...,17}
        \path (\x) edge (\lastx);
        
        \node[vert] (13) at (13,0) {};
        \draw (12) -- (13) -- (14);
        
        \node[vert] (18) at (13, -1) {};
        \node[vert] (19) at (13, -2) {};
        \node[vert] (20) at (13, -3) {};

        \draw (13) -- (18) -- (19) -- (20);
        \end{scope}

        \begin{scope}[yshift=-5cm]
        \foreach \x in {1,...,9}
        \node[vert] (\x) at (\x, 0) {};
        \foreach \x [remember=\x as \lastx (initially 1)] in {2,...,9}
        \path (\x) edge (\lastx);

        \foreach \x in {11,...,18}
        \node[vert] (\x) at (\x, 0) {};
        \foreach \x [remember=\x as \lastx (initially 11)] in {12,...,18}
        \path (\x) edge (\lastx);
        
        \node[vert] (10) at (10,0) {};
        \draw (9) -- (10) -- (11);
        
        \node[vert] (19) at (10, -1) {};
        \node[vert] (20) at (10, -2) {};

        \draw (10) -- (19) -- (20);
        \end{scope}
    \end{tikzpicture}
    \caption{Trees $T_1^{(1)}$ and $T_2^{(1)}$.}
    \label{fig:two-trees}
\end{figure}
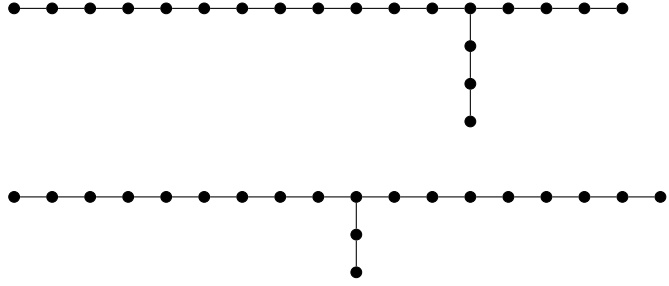

Both trees $T_1^{(k)}$ and $T_2^{(k)}$ have $20k$ vertices and three leaves, thus $\gp(T_1^{(k)}) = \gp(T_2^{(k)}) = 3$ by~\cite[Corollary 3.7]{manuel-2018a}. Observe that 
\begin{align*}
    \psi(T_1^{(k)}) & = 1 + 20 k x + \binom{20 k}{2} x^2 + 144 k^3 x^3,\\
    \psi(T_2^{(k)}) & = 1 + 20 k x + \binom{20 k}{2} x^2 + 144 k^3 x^3,
\end{align*}
where the coefficient at $x^3$ is obtained by taking one vertex from each pendent path. The key property that achieves the equality of polynomials is that $12 + 4 + 3 = 9 + 8 + 2$ and $12 \cdot 4 \cdot 3 = 9 \cdot 8 \cdot 2$. Note that this is not the only pair of triples with this property. 

\section{The general position polynomial of some graph operations}
\label{sec:operations}

In this section we consider the general position polynomials of disjoint unions of graphs, joins of graphs, and Cartesian products of graphs. 

Let $G\cupdot H$ denote the disjoint union of graphs $G$ and $H$. Then $S\subseteq V(G\cupdot H)$ is a general position set of $G\cupdot H$ if and only if $S\cap V(G)$ is a general position set of $G$ and $S\cap V(H)$ is a general position set of $H$. Using this fact, the following result readily follows. 

\begin{proposition}
\label{prop:disjoint-union}
If $G_1, \ldots, G_r$, $r\ge 2$, are graphs, then 
$$\psi(G_1 \cupdot \cdots \cupdot G_r) = \psi (G_1) \cdots \psi (G_r)\,.$$
\end{proposition}

Proposition~\ref{prop:disjoint-union} extends as follows. 

\begin{proposition}
\label{prop:product}
Let $G$ be a graph, $V_1,V_2$ a partition of $V(G)$, and  $G_1 = G[V_1]$, $G_2 = G[V_2]$. Then $\psi (G) = \psi (G_1)\psi (G_2)$ if and only if $G = G_1 \cupdot G_2$ or $G$ is complete. 
\end{proposition}

\begin{proof}
Sufficiency follows by Proposition~\ref{prop:disjoint-union} and Proposition~\ref{prop:examples}(i). 

Conversely, we need to prove that $\psi (G) = \psi (G_1)\psi (G_2)$ implies either that $G$ is a clique or that $G$ is the disjoint union of $G_1$ and $G_2$. Examine the subsets of order three of $V(G)$; in order to have equality in the $x^3$ term, we must have that every set of three vertices with two vertices in one of $G_1,G_2$ and one vertex in the other must be in general position. Suppose that there is an edge between $G_1$ and $G_2$ in $G$; we show that $G$ must be complete. Let $e$ be an edge from $G_1$ to $G_2$ with endpoints $u \in V(G_1)$ and $v \in V(G_2)$. Let $u^{\prime }$ be any neighbour of $u$ in $G_1$; as $\{ u,u^{\prime },v\} $ must be in general position, it follows that $v \sim u^{\prime }$. Inductively, we conclude that $u$ is adjacent to every vertex of $G_1$; furthermore, it follows that if $u_1$ and $u_2$ are non-adjacent vertices in $G_1$, then $u_1,v,u_2$ would not be in general position, so $G_1$ is a clique. Similarly $G_2$ must be a clique and we must have every edge between $G_1$ and $G_2$.  
\end{proof}

If $G$ and $H$ are disjoint graphs, then the {\em join} $G\vee H$ of $G$ and $H$ is the  graph with the vertex set $V(G\vee H) = V(G)\cup V(H)$, and the edge set $E(G \vee H) = E(G)\cup E(H)\cup \{xy:\ x\in V(G), y\in V(H)\}$.  Setting $\rho(G)$ to denote the maximum number of vertices in a union of pairwise independent cliques of $G$, it was proved in~\cite[Proposition 4.2]{ghorbani-2021} that $\gp(G \vee H )  = \max\{\omega(G) + \omega(H), \rho(G), \rho(H)\}$. 

The \emph{clique polynomial} $C(G)$ of a graph $G$ is the counting polynomial of cliques, that is, 
$$C(G) = c_0+c_1x+c_2x^2+\dots \,,$$ 
where $c_i$ is the number of cliques of order $i$ in $G$, cf.~\cite{hoede-1994}. Similarly, the \emph{independent union of cliques polynomial} $C_i(G)$ has coefficients equal to the number of independent union of cliques in $G$. Since a set $S \subseteq V(G_1 \vee G_2)$ is in general position if and only if either it induces a clique in both $G_1$ and $G_2$, or $S$ is an induced union of cliques in $G_1$ or $G_2$, the above discussion yields the following relation between the general position polynomial and the two clique polynomials. 

\begin{proposition}
If $G_1$ and $G_2$ are graphs, then 
\[\psi (G_1 \vee G_2) = (C(G_1)-1)(C(G_2)-1)+C_i(G_1)+C_i(G_2))-1.\]
\end{proposition}

We now turn our attention to the Cartesian product of graphs. Recall that the {\em Cartesian product} $G\cp H$ of graphs $G$ and $H$ has the vertex set  $V(G\cp H) = V(G)\times V(H)$ and the edge set $E(G\cp H)  = \{(g,h)(g',h'):\ gg'\in E(G)\mbox{ and } h=h', \mbox{ or, } g=g' \mbox{ and }  hh'\in E(H)\}$.

For the general position number of the Cartesian product of two paths we have (cf.~\cite{manuel-2018b}):

\begin{equation}
\label{eq:grids}
\gp(P_r \cp P_s) = \begin{cases}
    2; & r = s = 2,\\
    3; & r = 2, s \geq 3,\\
    4; & r, s \geq 3.
\end{cases}\end{equation}
Moreover, in~\cite[Theorem 2.1]{Klavzar-2021} it was proved that the number of general position sets in $P_r\cp P_s$ of cardinality $\gp(P_r\cp P_s)$ is equal to
\begin{equation}
\label{eq:grids-number-of}
\left\{
	\begin{array}{ll}
	\vspace*{2mm}
	6; & r = s = 2\,, \\
	\vspace*{2mm}
	\displaystyle{\frac{s(s-1)(s-2)}{3}}; & r = 2, s\ge 3\,, \\
	\displaystyle{\frac{rs(r - 1)(r - 2)(s - 1)(s - 2)(r(s - 3) - s + 7)}{144}}; & r, s\ge 3\,.
	\end{array}
	\right.
\end{equation}
From~\eqref{eq:grids-number-of} we immediately obtain the general position polynomial of thin grids: 

\begin{corollary}
    \label{thm:grids-small}
    If $r, s \geq 2$, then $$\psi(P_r\cp P_s) = \begin{cases}
        6x^2+4x+1; & r = s = 2,\\
        \frac{s(s-1)(s-2)}{3}x^3 + {2s \choose 2}x^2+2sx+1; & r = 2, s \geq 3.\\
    \end{cases}$$
\end{corollary}

For larger grids we have: 

\begin{theorem}
    \label{thm:grids-large}
    If $r, s \geq 3$, then 
    \begin{align*}
        \psi(P_r\cp P_s) & = 
        \frac{rs(r - 1)(r - 2)(s - 1)(s - 2)(r(s - 3) - s + 7)}{144} \; x^4 \\ & + \frac{1}{18}(r-1)r(s-1)s(r(2s-1)-s-4) \; x^3 \\ & + {rs \choose 2} x^2 + rs x + 1.
    \end{align*}
\end{theorem}

\begin{proof}
    Let $r,s\geq 3.$ It follows from~\eqref{eq:grids} that $\psi(P_r\cp P_s)$ is of degree 4. From~\eqref{eq:grids-number-of} we get the leading coefficient, while coefficients of $x^2, x^1$ and $x^0$ are obviously ${rs \choose 2}, rs$ and 1, respectively. Consider the number of general position sets with three vertices. There are ${rs \choose 3}$ 3-subsets of $V(P_r\cp P_s)$, but some of them are not in general position. In the following, we count the number of 3-subsets of $V(P_r\cp P_s)$ that are not in general position. 
    \begin{enumerate}
        \item There are $r {s \choose 3}$ and $s {r \choose 3}$ sets where all vertices are in the same horizontal or vertical layer.
        \item Consider the case where  exactly two are in the same horizontal layer (the case where they are in the same vertical layer is similar). Suppose that the second coordinate is $k$. The one with smaller first coordinate has $r-1$ possibilities; suppose it is $i$, while the one with greater coordinate can be in $\{i+1,\ldots,r\}$, say $j$. Since the triplet is not in general position, the third vertex can have any other second coordinate (any of $\{1,\ldots, k-1,k+1,\ldots,s\}$), and for its first coordinate $x$ either $x\leq i$ or $x\geq j$. Using the same reasoning for the case where two coordinates are in the same vertical layer and subtract the cases where both of them are the same, we obtain that in this case the number of sets that are not in general position is equal to
        \begin{align*}
        & r\sum_{i=1}^{s-1}\sum_{j=i+1}^s(s-(j-i-1))(r-1) + \\ & s\sum_{i=1}^{r-1}\sum_{j=i+1}^r(r-(j-i-1))(s-1) -  4 {r \choose 2} {s \choose 2}.
        \end{align*}
        \item The last case is vertices $(x_1,y_1)$, $(x_2,y_2)$ and $(x_3,y_3)$, where $x_1<x_2<x_3$ and either $y_1<y_2<y_3$ or $y_1>y_2>y_3$. There are: $$2\sum_{i=1}^{r-2}\sum_{j=1}^{s-2}\sum_{k=i+2}^{r}\sum_{l=j+2}^{s}(k-i-1)(l-j-1)$$ such sets.
    \end{enumerate}
     
    By subtracting from the number of all sets the number of sets that are not in general position, we get this simplified expression:    
    $$\frac{1}{18}(r-1)r(s-1)s(r(2s-1)-s-4),$$    
    representing the coefficient of the $x^3$ term in the general position polynomial.
\end{proof}


\section{Unimodality}
\label{sec:unimodality}


In this section we consider the unimodality of the general position polynomial. First, it is not unimodal in general as shown by the following example, which follows from Proposition~\ref{prop:examples}(v): 
$$\psi(K_{8,4}) = 1+12x^1+66x^2+60x^3+71x^4+56x^5+28x^6+8x^7+x^8\,.$$
Another complete bipartite graph for which the general position polynomial is not unimodal is $K_{9,7}$. 

In view of the above example and the situation with the independence polynomial, we can ask ourselves whether the general position polynomial is unimodal on trees. The answer is negative as we now demonstrate. 

A broom $B_{s,r}$, $s \geq 0$, $r \geq 0$, is a graph with vertices $u_0, \ldots, u_s, v_1, \ldots, v_r$, and edges $u_i u_{i+1}$ for $i \in \{0,\ldots, s-1\}$ and $u_0 v_j$ for $j \in [r]$. See Fig.~\ref{fig:broom} for an example.

\begin{figure}[htb]
    \centering
    \begin{tikzpicture}
    [scale=1,
    vert/.style={circle,fill=black,draw=black, inner sep=0.05cm},
    s/.style={fill=black!15!white, draw=black!15!white, rounded corners},
    outvert/.style={rectangle,draw=black},
    outedge/.style={line width=1.5pt},
    dom_vert/.style={circle,draw=black,fill=white, inner sep=0.05cm}
    ] 
        \begin{scope}
        \foreach \x in {0,...,4}
        \node[vert, label=below:$u_{\x}$] (\x) at (\x, 0) {};
        \foreach \x [remember=\x as \lastx (initially 0)] in {1,...,4}
        \path (\x) edge (\lastx);

        \foreach \x in {1,...,6}
        \node[vert, label=left:$v_{\x}$] (\x+10) at (-2, -1/2*\x+1/2*3.5) {};
        \foreach \x in {1,...,6}
        \path (\x+10) edge (0);
        
        \end{scope}
    \end{tikzpicture}
    \caption{The broom $B_{4,6}$.}
    \label{fig:broom}
\end{figure}
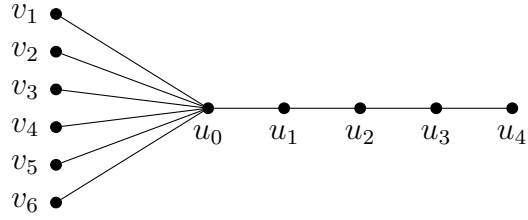

It is straightforward to check that 
\begin{align*}
\psi(B_{s,r}) & = \sum_{k \geq 0} b_k x^k \\
& = 1 + (s+r+1) x + \binom{s+r+1}{2} x^2 + \sum_{k \geq 3} \left( s \binom{r}{k - 1} + \binom{r}{k}\right)x^k\,.
\end{align*}
The smallest broom whose general position polynomial is not unimodal is $B_{17,6}$; its general position polynomial is
$$\psi(B_{17,6}) = 1 + 24 x + 276 x^2 + 275 x^3 + 355 x^4 + 261 x^5 + 103 x^6 + 17 x^7\,.$$
Moreover, one can calculate that if 
$$r \geq 6\quad {\rm and}\quad s \geq \left \lceil \frac{1}{2} \left(r^2-3 r-1\right)+\frac{\sqrt{3 r^4-14 r^3-3 r^2+14 r+3}}{2\sqrt{3}} \right \rceil\,,$$ 
then 
$b_1 < b_2 > b_3 < b_4$ holds, and hence there are infinitely many brooms for which the general position polynomial is not unimodal.

On the positive side, we first prove that the general position polynomial of comb graphs are unimodal. Recall that the {\em comb} $G_n$, $n\ge 1$, is obtained from the path $P_n$ by respectively attaching a pendent vertex to each of its vertices. 

\begin{theorem}
\label{thm:combs}
If $n\ge 1$, then $\psi(G_n)$ is unimodal.
\end{theorem}

\begin{proof}
    Let $\psi(G_n) = \sum_{k \geq 0} a_k x^k$. Clearly, $a_0 = 1$, $a_1 = 2n$, and $a_2 = \binom{2n}{2}$. For $k \geq 3$, we can determine $a_k$ by distinguishing between zero, one or two vertices from the general position set belonging to the path $P_n$ in $G$:
    \begin{align*}
        a_k & = \binom{n}{k} + \sum_{i=1}^n \left( \binom{i-1}{k-1} + \binom{n-i}{k-1} \right) + \sum_{i = 1}^{n-1} \sum_{j=i+1}^n \binom{j-i-1}{k-2}\\
         & = \frac{1}{k} \left( \frac{(n-1) n}{k-1} \binom{n-2}{k-2}+n \binom{n-1}{k-1}+(n-k+1) \binom{n}{k-1}+k \binom{n}{k} \right).
    \end{align*}
    In particular, $a_3 = \frac{2}{3} (n-2) (n-1) n$, and $a_3 \geq a_2$ if and only if $n \geq 6$. 

    If $n = 1$, then $G_1 = K_2$, while if $n = 2$, then $G_2 = P_4$, and their polynomials are unimodal. For $n \in \{3,4,5\}$, we calculate the polynomials explicitly to check that they are also unimodal:
    \begin{align*}
        \psi(G_3) & = 4 x^3+ 15 x^2+6 x+1\\
        \psi(G_4) & = 4 x^4+16 x^3+ 28 x^2+8 x+1\\
        \psi(G_5) & = 4 x^5+20 x^4+40 x^3+ 45 x^2+10 x+1
    \end{align*}

    For $n \geq 6$, we already know that $a_0 \leq a_1 \leq a_2 \leq a_3$, thus it only remains to show that the sequence $(a_k)_{k \geq 3}$ is unimodal. 
    The difference between two terms in $\psi(G)$ for $3 \leq k \leq n-1$ can be simplified as follows:
    $$a_k - a_{k+1} = \frac{4 n! (2 k-n+1)}{(k+1)! (n-k)!},$$
    which implies that for $3 \leq k \leq \frac{n-1}{2}$, $a_k \leq a_{k+1}$, and that for $\frac{n-1}{2} \leq k \leq n-1$, $a_k \geq a_{k+1}$, so $(a_k)_{k \geq 3}$ is indeed unimodal.
\end{proof}

Another family of graphs for which the general position polynomial is unimodal is the class of Kneser graphs $K(n,2)$. Recall that the vertex set of $K(n,2)$ contains all $2$-subsets of an $n$-set, two vertices being adjacent if the corresponding sets are disjoint. At the end of Section~\ref{sec:basics} we have considered the special case $P = K(5,2)$. 

\begin{theorem}
    \label{thm:kneser}
    If $n\ge 2$, then $\psi(K(n,2))$ is unimodal.
\end{theorem}

\begin{proof}
We first determine the general position polynomial of $K(n,2)$. Recall from~\cite{manuel-2018a} that $$\gp(K(n,2)) = \begin{cases}
    6; & n \leq 6,\\
    n-1; & n \geq 7.
\end{cases}$$
General position sets of size $j \in \{2, \ldots, n-1\}$ can be cliques or independent sets, and for $j \geq 7$ this is the only possibility. To form a clique on $j$ vertices in $K(n,2)$, select $2j$ elements of the $n$-set (this can be done in $\binom{n}{2j}$ ways) and put them into unordered pairs ($\binom{2j}{2,\ldots,2} \frac{1}{j!} = \frac{(2j)!}{2^j j!}$ options). An independent set on $j$ vertices can be of the form $\{ax_1, \ldots, ax_j\}$, where $a, x_1, \ldots, x_j$ are distinct elements of the $n$-set. There are $n \binom{n-1}{j}$ such sets. For $j \geq 4$, all independent sets are of this form. However, for $j = 3$, independent sets can also take the form $\{ab, bc, ac\}$, where $a,b,c$ are distinct elements of the $n$-set. For $j \in \{ 3, \ldots, 6\}$, several additional types of general position sets are possible. On six vertices, $3K_2$ forms a general position set, and there are $\binom{n}{4}$ such sets (they are of the form $\{ab, cd, ac, bd, ad, bc\}$). On five vertices, $2K_2 \cupdot K_1$ is also a general position set, and it can be obtained by removing one vertex from the general position set on six vertices, thus there are $6 \binom{n}{4}$ of them. Similarly, on four vertices, $2K_2$ or $K_2 \cupdot 2K_1$ are also independent sets. They are obtained by removing two vertices from the general position set on six vertices, so there are $\binom{6}{2} \binom{n}{4} = 12 \binom{n}{4}$ of them. On three vertices, we need to consider the additional type of independent sets as well, and there are $\binom{n}{3}$ of them. By removing three vertices from a general position set on six vertices we can also obtain a set of three vertices in general position that induces a copy of $K_2 \cupdot K_1$, but these vertices must not belong to three copies of different $K_2$. Thus there are $\left( \binom{6}{3} - 2^3 \right) \binom{n}{4} = 12 \binom{n}{4}$. Therefore: 
\begin{align*}
    \psi(K(n,2)) & = \sum_{k = 0}^{n-1} a_k x^k \\
    & = 1 + \binom{n}{2} x + \left( \binom{n}{3} + 12 \binom{n}{4} \right) x^3 + 15 \binom{n}{4} x^4 + 6 \binom{n}{4} x^5 + \binom{n}{4} x^6  \\
    &\quad + \sum_{j=2}^{n-1} \left( \binom{n}{2 j} \frac{ (2 j)!}{2^{j} j!}+n \binom{n-1}{j} \right) x^j\,.
\end{align*}

We can check by computer that $\psi(K(n,2))$ is unimodal for $2 \leq n \leq 16$. For $n \geq 17$, the following inequalities hold: $a_0 \leq a_1 \leq \cdots \leq a_6 \leq a_7$. Thus it remains to show that $(a_k)_{k \geq 7}$ is unimodal. Since $k \geq 7$, $a_k = n \binom{n-1}{k} + \frac{(2k)!}{2^k k!} \binom{n}{2k}$. To show unimodality, we simplify the difference 
\begin{align*}
a_k - a_{k+1} = & \frac{\binom{n}{k+1}}{2^{k+1}} \Bigl( 2^{k+1} (2k+2-n) \\
& + (n+2-(n-2k)^2) (n-2k+1)(n-2k+2) \cdots (n-k-1) \Bigr)\,.
\end{align*}
To determine for which $k$ the difference $a_k - a_{k+1} \geq 0$ and for which $k$ it is $a_k - a_{k+1} \leq 0$, we only need to consider the terms
\begin{align*}
A & = 2^{k+1} (2k+2-n), \\
B & = (n+2-(n-2k)^2) (n-2k+1)(n-2k+2) \cdots (n-k-1)\,.    
\end{align*}
First observe that if $k > \frac{n}{2}$, then $\frac{(2k)!}{2^k k!} \binom{n}{2k} = 0$, thus $a_k - a_{k+1} = \frac{\binom{n}{k+1}}{2^{k+1}} \cdot A > 0$. 

If $\frac{n}{2} - 1 \leq k \leq \frac{n}{2}+1$, then $A \geq 0$, and since $n+2-(n-2k)^2 \geq 0$ and for all $j \in [k-1]$, $n-2k+j \geq 0$, we also have $B \geq 0$. Thus $a_k - a_{k+1} \geq 0$.

For $\frac{n}{2} - \frac{\sqrt{n+1}}{2} \leq k \leq \frac{n}{2} - 1$, we have $A \leq 0$ and $B \geq 0$. In the following we will prove that $B \geq |A|$. Observe that $n+2-(n-2k)^2 \geq 1$, $n-2k+1 > |2k+2-n|$ and for all $j \in \{2,\ldots, k-1\}$, $n-2k+j \geq 4$. Thus $$B \geq 1 \cdot |2k+2-n| \cdot 4^{k-2} > |2k+2-n| \cdot 2^{k+1} = |A|$$ since $k \geq 7$. Hence we have $a_k - a_{k+1} \geq 0$. 

For $7 \leq k \leq \frac{n}{2} - \frac{\sqrt{n+2}}{2}$, we have $A < 0$, $n+2-(n-2k)^2 \leq 0$ and for all $j \in [k-1]$, $n-2k+j \geq 0$, thus $B \leq 0$. It follows that $a_k - a_{k+1} < 0$. 

It remains to prove that the above cases indeed cover all integers $k$, $7 \leq k \leq n-1$. To see this we need to prove that no integer lies in the interval $\left( \frac{n}{2} - \frac{\sqrt{n+2}}{2}, \frac{n}{2} - \frac{\sqrt{n+1}}{2} \right)$. Suppose that there exists $m \in \mathbb{N}$ such that $\frac{n}{2} - \frac{\sqrt{n+2}}{2}< m < \frac{n}{2} - \frac{\sqrt{n+1}}{2}$. Simplifying and squaring this chain of inequalities yields $n+1< (n-2m)^2 < n+2$. But since $n+1$ and $n+2$ are consecutive integers, we obtain a contradiction. Thus we have proved that $(a_k)_{k\geq7}$ is unimodal.
\end{proof}

The following family of graphs $T^*(r,a)$ from~\cite{TuiThoCha} yields another family of graphs with unimodal general position polynomial. Take a complete $a$-partite graph, each part of which contains $r$ vertices and label the vertices in part $i$ by $i_1,\ldots ,i_r$. Then for each $i\in [r]$, delete the edges of the clique induced by the vertices $i_1, \ldots, i_a$. 

Now suppose that $S$ is a general position set of $T^*(r,a)$. Any subset lying in a single partite set is in general position. Suppose that $S$ contains three vertices (say with labels 1,2,3) in part $A$; then $S$ can contain no vertices from other partite sets, for when we add a new vertex $x$ from another part, the label of $x$ will differ from that of at least two vertices of $S$. Suppose then that $S$ contains two vertices $a_1,a_2$ of a part $A$ (say with labels 1,2); by the same reasoning $S$ can only contain vertices with labels 1 and 2. If $S$ intersects only two parts, it will be in general position, inducing either a $K_2 \cupdot K_1$ or a $2K_2$. However, $S$ cannot contain vertices from more than two partite sets; if $S$ contains a vertex $b_1$ with label 1 in a part $B$ and a vertex $c_1$ with label $1$ in a part $C$, then $b_1,a_2,c_1$ is a shortest path, whereas if $S$ contains a vertex $b_1$ in $B$ with label 1 and a vertex $c_2$ with label $2$ in $C$, then $b_1,c_2,a_1$ would again be a shortest path. It follows that the general position polynomial of this graph is given by
\begin{align*}
\psi (T^*(r,a)) = & 1+nx+{n \choose 2}x^2+2a(a-1){r \choose 2}x^3+{a \choose 2}{r \choose 2}x^4 \\
& +\sum _{i\geq 3}\left [ a{r \choose i}+r^i{a \choose i}\right ]x^i\,,
\end{align*}
where $n = ra$.

\begin{proposition}
    \label{prop:T*-unimodal-partial}
    If $a \in \{1,2\}$, then $T^*(r,a)$ is unimodal.
\end{proposition}

\begin{proof}
    If $a = 1$, then $T^*(r,1) = K_r$, which is unimodal. If $a=2$, then $T^*(r,2)$ is a complete bipartite graph without a perfect matching. Simplifying its general position polynomial gives 
    \[ \psi (T^*(r,2)) = 1+2rx+{2r \choose 2}x^2+4{r \choose 2}x^3+{r \choose 2}x^4+\sum _{i\geq 3} 2{r \choose i}x^i.\]
    Observe that the sequence $(2 \binom{r}{i})_{i \geq 3}$ is unimodal. Thus if we can prove that the initial coefficients of $\psi(T^*(r,2))$ are increasing, the general position polynomial is also unimodal. This holds for $r \geq 10$, as $1 \leq 2 r \leq \binom{2 r}{2} \leq 4 \binom{r}{2} + 2 \binom{r}{3} \leq  \binom{r}{2} + 2 \binom{r}{4} \leq 2 \binom{r}{5}$ holds for $r \geq 10$. 

    \begin{table}[htb]
        \centering
        \begin{tabular}{c|l}
            $r$ & $\psi(r,2)$ \\ \hline
             $1$ & $x^2+2 x+1$ \\
             $2$ & $x^4+4 x^3+6 x^2+4 x+1$ \\
             $3$ & $3 x^4+14 x^3+15 x^2+6 x+1$ \\
             $4$ & $8 x^4+32 x^3+28 x^2+8 x+1$ \\
             $5$ & $2 x^5+20 x^4+60 x^3+45 x^2+10 x+1$ \\
             $6$ & $2 x^6+12 x^5+45 x^4+100 x^3+66 x^2+12 x+1$ \\
             $7$ & $2 x^7+14 x^6+42 x^5+91 x^4+154 x^3+91 x^2+14 x+1$ \\
             $8$ & $2 x^8+16 x^7+56 x^6+112 x^5+168 x^4+224 x^3+120 x^2+16 x+1$ \\
             $9$ & $2 x^9+18 x^8+72 x^7+168 x^6+252 x^5+288 x^4+312 x^3+153 x^2+18 x+1$ \\
        \end{tabular}
        \caption{General position polynomials for $a = 2$ and small values of $r$.}
        \label{tab:a2-rsmall}
    \end{table}
    
    The unimodality of the remaining cases can be checked by hand, see Table~\ref{tab:a2-rsmall}. 
\end{proof}

\section{Concluding remarks}
\label{sec:conclude}

Recall that the {\em corona} $G\circ K_1$ of a graph $G$ is obtained from $G$ by attaching a pendent vertex to each the vertices of $G$. We wonder whether the following extension of Theorem~\ref{thm:combs} holds: 

\begin{problem}
Assume that $\psi(G)$ is unimodal. Then is $\psi(G\circ K_1)$ also unimodal?
\end{problem}

In Proposition~\ref{prop:T*-unimodal-partial} we have proved that if $a \in \{1,2\}$, then $T^*(r,a)$ is unimodal. This leads to: 

\begin{problem}
For which pairs $(r,a)$ is the graph $T^*(r,a)$ unimodal? 
\end{problem}

For example, $T^*(r,3)$ is unimodal if $r \geq 19$, but also for some smaller values of $r$. 

Several variations of the general position number have been investigated in the literature. For example, a set $S \subseteq V(G)$ is in \emph{monophonic position} if no induced path of $G$ contains three vertices of $S$ (see~\cite{ThoChaTuiSte}), whilst $S$ is a \emph{mutual-visibility set} if for any $u,v \in S$ there exists a shortest $u,v$-path in $G$ that does not pass through $S\setminus \{ u,v\} $ (see~\cite{DiStefano}). We suggest than an interesting direction for future research would be to explore the polynomials counting such sets and their relation to the general position polynomials.

\section*{Acknowledgements}

This work was supported by the Slovenian Research Agency ARIS (research core funding P1-0297 and projects J1-2452, N1-0285, Z1-50003). Vesna Ir\v{s}i\v{c} also acknowledges the financial support from the ERC KARST project (101071836). The research of James Tuite was partially funded by LMS Research in Pairs grant number 42235; he also thanks the University of Ljubljana for their hospitality. The authors also thank Elias John Thomas for suggesting this topic.

\section*{Declaration of interests}
 
The authors declare that they have no conflict of interest. 

\section*{Data availability}
 
Our manuscript has no associated data.


\end{document}